\documentclass[12pt]{amsart}
\usepackage{amssymb}
\usepackage{amsfonts}
\usepackage{amssymb,latexsym}
\usepackage{enumerate}
\usepackage{mathrsfs}

\usepackage{graphicx}
\usepackage[all]{xy}

\newtheorem{thm}{Theorem}[section]

\newtheorem{prop}[thm]{Proposition}
\newtheorem{mprop}[thm]{Main Proposition}

\newtheorem{rem}[thm]{Remark}

\theoremstyle{definition}
\newtheorem{definition}[thm]{Definition}

\theoremstyle{remark}

\begin{document}

\title[{\normalsize {\large {\normalsize {\Large {\LARGE }}}}}A remark on  extended  Kim's conjecture  and  Hypo-Lie algebra
]{A remark on extended  Kim's conjecture  and Hypo- Lie algebra}

\author{Kim YangGon, Wang MoonOk }

\address{emeritus professor,
	 Department of Mathematics,  
	Jeonbuk National University, 567 Baekje-daero, Deokjin-gu, Jeonju-si,
	Jeollabuk-do, 54896, Republic of Korea.}

\

\

\email{ kyk1.chonbuk@hanmail.net    
wang@hanyang.ac.kr }


\begin{abstract}
We have already conjectured 2 important guesses   regarding  Hypo- Lie algebra and modular simple Lie algebra. We would like to  attach 2  important guesses  more to this conjecture. Such new guesses are related to the Steinberg module.

\end{abstract}

\maketitle

\section{introduction}

\

\

\large{Let $L$ be any modular simple Lie algebra of $A_l$-type or $C_l$-type over any algebraically closed field  $F$ of characteristic $p\geq 7$.\newline

We proved that Kim's conjecture is right for these simple Lie algebras. We still believe that Kim's conjecture is  right for other simple Lie algebras of classical type.\newline

 Furthermore 2 more guesses are plausible attached to the conjecture.\newline

In section 2  we recollect some definitions and related facts to the conjecture. We shall give additional guesses to this conjecture in section 3.

\

\section{ Some definitions and related facts}

\

\

Now let $L$ be any simple Lie algebra of classical type with a CSA $H$ over an algebraically closed field  $F$ of characteristic $p\geq 7$.\newline

 It is well known that  any simple $L$-module is isomorphic to some quotient  $L$-module of $V_\chi^\lambda (L)$,which is called a  $Verma$ module. Here $\lambda$ is  a weight  $\lambda:H \rightarrow F$  which is  related to  $\chi$  as a  linear form in $H^\ast$. \newline

Such a quotient  $L$-module is  called a $Weyl$ module, denoted by $W_\chi^\lambda (L)$.\newline

We posed a conjecture in [5] and  proved  completely in  [2]  that the conjecture is right  for  modular simple  Lie algebras of  $A_l$-type or $C_l$-type.\newline

We still believe that  our conjecture is also right  for other modular simple Lie algebras of classical type.
Furthermore we would like to announce  to  the  interested readers  the additional guesses attached to the conjecture.

\begin{definition}

Suppose that  $V_0^\lambda (L)= W_0^\lambda (L)$ and $\rho_0^\lambda $  is its associated  representation of $U(L)$; then this simple module is said to be a $Steinberg$ module, while  $U(L)/ker \rho_0^\lambda$ is called  a  $Steinberg$ algebra, which is isomorphic to $\rho_0^\lambda (U(L))$ associated with the Steinberg module.

\end{definition}

For a root $\alpha$ in a root system $\Phi$ of $L$ \,we put $w_\alpha:= (h_\alpha + 1)^2+ x_{-\alpha}x_\alpha$ and  $g_\alpha:= x_\alpha^{p-1}- x_{-\alpha}$ as in [0],[2],[3],[4],[5],[6]. \newline

By virtue of  [8], we know that $w_\alpha$ belongs to the center of  $U(\frak{sl}_2(F))$, and  that $g_\alpha$ is invertible  in  $\rho_0^\lambda (U(L))$ whenever  $x_\alpha^p= x_{-\alpha}^p \equiv 0$ but $h_\alpha^p- h_\alpha \not\equiv 0$ modulo $ker \rho_0^ \lambda$.

\begin{prop}
Let  L be  any modular simple  Lie algebra of $A_l$- type or of  $C_l$- type  over an algebraically closed field  $F$ of characteristic $p\geq 7$.\newline

Suppose that  $\chi= 0$ and  there exists  a root  $\alpha$ such that  $\lambda (h_\alpha)= -1$; then we have a  Steinberg module $V_0^\lambda (L)= W_0^\lambda (L)$, where  $\lambda (h_\alpha^p)- \lambda (h_\alpha)= \chi (h_\alpha)^p$ holds.

\end{prop}

\begin{proof}

We know that  $w_\alpha$ acts on   maximal vectors  as  a constant zero  and so on factor $S_\alpha$- modules relative to a composition series in $V_0^\lambda(L)$, where  $S_\alpha= Fx_\alpha+ Fh_\alpha+ Fx_{-\alpha} \simeq \frak{sl}_2(F)$.  \newline

In  such cases $g_\alpha$  becomes invertible  in $U(L)/ker \rho_0^\lambda $. Hence the proofs required  along with  Lee's bases are exactly the same as those in [2].\newline

For $A_l$-type, refer to  theorem2.2 , and for  $C_l$-type refer to proposition3.2  resectively in  the reference [2].

\end{proof}

\

\section{ Conjecture extended from [5]}

\

\

In the main proposition of  the reference [0] we made use of a  Steinberg module for  simple Lie algebra $L$ of $B_l$-type , which  arises seemingly unique.\newline

For the proof of simplicity of Steinberg module we gave  directly  Lee's base of  Steinberg algebra $U(L))/ker \rho_0^{-\lambda}$, where $\lambda$ is a Weyl weight.\newline

We found out that Steinberg modules  may still arise for  any  modular simple  Lie  algebra  $L$ of classical type  over $F$ even if $\lambda$ is not a Weyl weight.\newline

The reason goes as follows.\newline

We are well aware that  $U(L)$ is integral over  its center $\frak Z(U(L))$, so that  $w_\alpha$ has an irreducible integral equation of degree $p$  over  $\frak Z(U(L))$ .\newline

If  $Fx_\alpha+ Fh_\alpha+ Fx_{-\alpha}= \frak {sl}_2(F)$,then this irreducile polynomial over the center of $U(\frak{sl}_2(F))$  has the form\newline

$\Pi_{k=0}^{p-1}(t- k^2)- z^2- 4xy $,\newline

where\newline

 $x= x_{-\alpha}^p$,\newline

$y= x_{\alpha}^p$,\newline

$z= h_\alpha^p- h_\alpha$,\newline

$t= (h_\alpha+ 1)^2 - 4x_{-\alpha}x_\alpha$,\newline

$k= 0,1,2,\cdots ,p-1$
( refer to[8]).\newline

Let   $\{\lambda_i|1\leq i \leq l\}$ denote  the set of  fundamental dominant weights  and  let  $\lambda$ be the Weyl weight for the time being.
In other words $\lambda= \Sigma_{i=1}^l \lambda_i= \Sigma_{\alpha \succ 0} \alpha$.\newline

  Then for any  root $\alpha_j$ with $1\leq j \leq l$ we have \newline

$\lambda (h_{\alpha_j})= <\lambda,\alpha_j>= \{\Sigma_i 2(\lambda_i, \alpha_j)\}/(\alpha_j, \alpha_j)= \Sigma_i \delta_{ij}= 1$.
We thus obtain $(-\lambda)(h_{\alpha_j})= -1$.\newline

Putting $S_{\alpha_j}= Fx_{\alpha_j}+ h_{\alpha_j}+ Fx_{-\alpha_j}$,which is isomorphic to $\frak {sl}_2(F)$, we know that $w_{\alpha_j}$ acts as 0 on factor $S_{\alpha_j}$-modules relative to a composition series in  $W_0^{-\lambda}(L)$ since every Weyl module has a maximal vector.\newline

In other words, $h_{\alpha_j}$ acts on a maximal vector as -1.
According to [8],  $g_{\alpha_j}$  for  $B_l$-type $L$ is invertible in $W_0^{-\lambda}(L)$  because  $w_{\alpha_j}$  is  nilpotent  in this  Steinberg algebra. Note that  $w_\alpha$  becomes either nilpotent or invertible in this Steinberg algebra. So we could make  in  the reference [0]  a Lee's basis of its  associated Steinberg algebra by way of  $g_\alpha$.\newline

For  the purpose of  making a Lee's basis  for  $\chi= 0$, it is necessary  to get  $g_\alpha$ invertible in $W_0^{-\lambda}(L)$.
Hence we  might  as well  justify our  extended conjecture in general  motivated by  this idea.\newline

$[\textbf{extended CONJECTURE }]$ \newline

In addition to guesses (i),(ii) in the conjecture in [5], we might give 2 more guesses (iii) and (iv) explained below.\newline

(iii)Let $L$ be any modular simple Lie algebra of classical type over an algebraically closed field  $F$ of chracteristic $p\geq 7$; then we conjectured  in [5] that $V_\chi^\lambda (L)= W_\chi^\lambda (L)$  and  now we conjecture even more that there exists  a  Lee's basis whenever $\chi \neq 0$ , where $\lambda$ is  a weight $\lambda: H \rightarrow F$ relating to a  maximal vector of  the Weyl module.

However even if $\chi=0$ , we  also conjecture  that  we might give Lee's basis to  the Steinberg  algebra associated with $W_0^{-\lambda}(L)$  for a  Weyl  weight $\lambda$ .\newline

(iv)We conjecture that  A  simple $L$-module over $F$ of dimension less than $p^m$ arises with $2m+ rank(L)= dim (L)$  if and only if  $\chi= 0$ and $\lambda (h_\alpha)\neq -1$ for any root $\alpha \in$ the base $\Delta$ of the root system $ \Phi$ of $L$.\newline

So combining (iii) and (iv) we guess that there might be  only  a finite number of  simple  $L$ modules of dimension  $<p^m $ up to isomorphism and hence 

$L$ is a  $Hypo$- Lie algebra.

\

\

\bibliographystyle{amsalpha}

\end{document}